\documentclass[12pt]{amsart}
\usepackage{amssymb}

\usepackage{hyperref}

\numberwithin{equation}{section}

\newtheorem{theorem}[equation]{Theorem}
\newtheorem{lemma}[equation]{Lemma}

\newcommand{\abs}[1]{\lvert#1\rvert} 
\newcommand{\Ad}{\operatorname{Ad}} 
\newcommand{\ad}{\operatorname{ad}} 
\newcommand{\eps}{\epsilon}
\newcommand{\cgp}{G} 
\newcommand{\step}{s} 
\newcommand{\calg}{\mathfrak{g}} 
\newcommand{\calglayer}[1]{V_{#1}} 
\newcommand{\dil}[1]{\delta_{#1}} 
\newcommand{\brkt}[2]{[#1,#2]} 
\newcommand{\conjg}[1]{\operatorname{C}_{#1}} 
\newcommand{\lquot}{G/H} 
\newcommand{\expconjg}[2]{\conjg{\exp\left(#1\right)}\left(\exp(#2)\right)} 
\newcommand{\perturbvec}{Y} 
\newcommand{\verterr}{h} 
\newcommand{\errvec}{Z} 
\newcommand{\errveccomp}{W} 

\newcommand{\lowestordercoeffcondition}{-(2-D)<0}
\newcommand{\conjgdot}{\cdot} 
\newcommand{\defeq}{=} 
\newcommand{\subriemannian}{subriemannian } 

\begin{document}
\title{Non-minimality of corners in \MakeUppercase{\subriemannian}geometry}

\author[Eero Hakavuori]{Eero Hakavuori} 
\email{eero.j.hakavuori@jyu.fi}

\author[Enrico Le Donne]{Enrico Le Donne}
\email{enrico.ledonne@jyu.fi}

\address[Hakavuori and Le Donne]{
Department of Mathematics and Statistics,  
University of Jyv\"askyl\"a, 40014 Jyv\"askyl\"a, Finland}


\renewcommand{\subjclassname}{%
\textup{2010} Mathematics Subject Classification}

\date{March 8, 2016}

\begin{abstract}
We give a short solution to one of the main open problems in \subriemannian geometry.
Namely, we prove that length minimizers do not have corner-type singularities.
With this result we solve Problem II of Agrachev's list, and provide the first general result toward the 30-year-old open problem of regularity of \subriemannian geodesics.
\end{abstract}

\maketitle
\tableofcontents

\section{Introduction}
One of the major open problems in \subriemannian geometry is the regularity of length-minimizing curves (see \cite[Section 10.1]{Montgomery} and 
\cite[Section 4]{Monti_survey}). This problem has been open since the work of Strichartz \cite{Strichartz, Strichartz2} and Hamenst{\"a}dt \cite{Hamenstadt}.

Contrary to Riemannian geometry, where it is well known that all length minimizers are $C^\infty$-smooth, the problem in the \subriemannian case is significantly more difficult. The primary reason for this difficulty is the existence of abnormal curves (see \cite{Agrachev_Sachkov,Agrachev_Barilari_Boscain:book}), which we know may be length minimizers since the work of Montgomery \cite{Montgomery1994}. 
Nowadays, many more abnormal length minimizers are known \cite{BryantHsu, Liu-Sussman2, Liu-Sussman, Gole_Karidi, Sussmann_cornucopia}.

Abnormal curves, when parametrized by arc-length, need only have Lipschitz-regularity (see \cite[Section 5]{LLMV2}), which is why, a priori, no further regularity can be assumed from an arbitrary length minimizer in a \subriemannian space.
However,  a recent result of Sussmann states that in the analytic setting, every length minimizer  is analytic on an open dense subset of its domain, see \cite{Sussmann:regularity}.
Nonetheless, even including all known abnormal minimizers, no example of a non-smooth length minimizer has yet been shown.

A considerable effort has been made to find examples of non-smooth minimizers (or to prove the non-existence thereof) in the simple case of curves where the lack of continuity of the derivative is at a single point. Partial results for the non-minimality of corners can be found, e.g., in \cite{Monti-nonminimizing, Monti_CV, LLMV}.

In this paper, we prove the non-minimality of curves with a corner-type singularity in complete generality. 
Thus we solve Problem II of Agrachev's list of open problems in \subriemannian geometry \cite{Agrachev_problems}, by proving the following result (definitions are recalled in Section~\ref{sec:def}):

\begin{theorem}\label{thm:mflds}
	Length-minimizing curves in \subriemannian manifolds do not have corner-type singularities.
\end{theorem}

In fact, our proof also shows that the same result holds even if instead of \subriemannian manifolds, we consider the slightly more general setting of Carnot-Carath\'eodory spaces with strictly convex norms.

\subsection{The idea of the argument}
The argument builds on ideas of the two papers \cite{Leonardi-Monti, LLMV_corner}. 
Up to a desingularization, blow-up, and reduction argument, it is sufficient to consider the case of a corner in a Carnot group of rank 2.
For Carnot groups, we prove the result (Theorem~\ref{thm:Carnot}) by induction on the step $s$ of the group, starting with $s=2$, i.e., the Heisenberg group, see Lemma~\ref{lem:starting:induction}.

For an arbitrary step $s\geq 3$ we project the corner into a Carnot group of step $s-1$. The inductive argument then gives us the existence of a shorter curve in the group of step $s-1$.
Lifting this curve back to the original group, we get a curve shorter than the initial corner, but with an error in the endpoint by an element of degree $s$, see Lemma~\ref{lem:lift:induction}.

We correct the error by a system of curves placed along the corner. 
In fact, we prove that this is possible with a system of three curves with endpoints in the subspace of degree $s-1$.
This last fact is the core of the argument (see Lemma~\ref{lem:technic}) and is a crucial consequence of the fact that the space is a nilpotent and stratified group.
 
Finally, we consider the situation at smaller scales by modifying the initial corner using an $\eps$-dilation of the lifted curve and suitable dilations of the three correcting curves.
By Lemma~\ref{lem:technic} the suitable factor to correct the error of the dilated corner-cut is $\epsilon^{s/(s-1)}$, essentially due to the fact that the error scales with order $s$ and the correction scales with order $s-1$.
Hence, the length of the new curve is the length of the corner plus a term of the form
\[-a \eps + b \epsilon^{s/(s-1)},\]
for some positive constants $a, b$. We conclude that for $\eps$ small enough the new curve is shorter than the corner.

\subsection{Definitions}\label{sec:def}
Let $M$ be a smooth Riemannian manifold and $\Delta$ a smooth subbundle of the tangent bundle. We consider the length functional $L_\Delta$ on curves in $M$ that for a curve $\gamma$ is defined as the Riemannian length of $\gamma$ if $\dot\gamma \in\Delta $ almost everywhere, and $\infty$ otherwise.
Analogously to the Riemannian setting, let $d_\Delta$ be the distance associated to $L_\Delta$. 
We assume that $\Delta$ is bracket generating, in which case $d_\Delta$ is finite and its length functional is $L_\Delta$. In this paper, we call $(M,d_\Delta)$ a {\em \subriemannian manifold}. For more on the subject see \cite{Gromov1, Gromov, Montgomery, jeancontrol, Rifford:book, Agrachev_Barilari_Boscain:book}.
If instead of a Riemannian structure, we use a continuously varying norm on the tangent bundle, we call the resulting metric space a {\em Carnot-Carath\'eodory space} ({\em C-C space}, for short).

Let $\gamma:[-1,1]\to M$ be an absolutely continuous curve on a manifold $M$.
We say that $\gamma$ has a {\em corner-type singularity} at time 0, if the left and right derivatives at $0$ exist and are linearly independent. 
 
Let $G$ be a Lie group. We say that a curve $\gamma:[-1,1]\to G $ is a {\em corner} if there exist linearly independent vectors $X_1, X_2$ in the Lie algebra of $G$ such that 
\[
\gamma(t) = \begin{cases}
\exp(-t X_1) & \text{if }t\in [-1,0]\\ 
\exp(t X_2) & \text{if }t\in (0,1]\,.
\end{cases}
\] 
In such a case, we will say that $\gamma$ is the corner from $\exp( X_1)$ to $\exp( X_2) $. Notice that at 0 the left derivative of $\gamma$ is $-X_1$, while the right derivative is $X_2$. Hence, a corner has a corner-type singularity at 0.

Let $G$ be a simply connected Lie group with a Lie algebra $\calg$ admitting a stratification, i.e., $\calg = \calglayer{1}\oplus\dots\oplus\calglayer{s}$, where $\calglayer{j}\subset\calg$ are disjoint vector subspaces of the algebra, such that $\calglayer{j+1} = \brkt{\calglayer{1}}{\calglayer{j}}$ for all $j=1,\dots,s$ with $\calglayer{s+1}= \{0\}$. 
The subspaces $\calglayer{j}$ are called the \emph{layers} of the stratification. Let $\abs{\cdot}$ be a norm on the first layer $\calglayer{1}$ of the Lie algebra. The Lie group $G$ together with a stratification $\calg = \calglayer{1}\oplus\dots\oplus\calglayer{s}$ of its algebra and a norm $\abs{\cdot}$ on the first layer $\calglayer{1}$ is called a \emph{Carnot group}. See \cite{Montgomery, LeDonne:Carnot} for more discussion on Carnot groups.

A Carnot group has a natural structure of C-C space where the subbundle $\Delta$ is the left-translation of the first layer $\calglayer{1}$ and the norm is extended left-invariantly.
Then by construction the C-C distance $d=d_\Delta$ on a Carnot group is left-invariant. 
In addition, a Carnot group also has a family of Lie group automorphisms $\{\dil{\epsilon}\}_{\epsilon>0}$ adapted to the stratification. Namely, each $\dil{\epsilon}$ is 
determined by $(\dil{\epsilon})_* (X)= \eps^j X$, for $X\in V_j$. 
Moreover, each map $\dil{\epsilon}$ behaves as an $\eps$-dilation for the C-C distance, i.e., $d(\dil{\epsilon}(g),\dil{\epsilon}(h)) = \epsilon d(g,h)$, for all points $g,h\in G$.

In a Carnot group, the curves $t\mapsto \exp(tX)$, with $X\in \calg$, have locally finite length if and only if $X\in \calglayer{1}$. 
Actually, such curves are length minimizing and $d(e,\exp(X))=\abs{X}$, where $e$ denotes the identity element of $G$.

A norm $\abs{\cdot}$ is {\em strictly convex} if in its unit sphere there are no non-trivial segments. Equivalently, if $\abs{x}=\abs{y}=1$ and $\abs{x+y}=2$, strict convexity implies $x=y$. 

\section{Preliminary lemmas}
The following lemma is the base of our inductive argument. In particular, it proves Theorem~\ref{thm:mflds} for the Heisenberg group equipped with a strictly convex norm.
\begin{lemma}
\label{lem:starting:induction}
Let $G$ be a step-2 Carnot group with a distance $d$ associated to a strictly convex norm.
Then in $(G,d)$ no corner is length minimizing.
\end{lemma}
\begin{proof}
Let $X_1$ and $X_2$ be linearly independent vectors of the first layer $\calglayer{1}$ of $G$.
For $\eps>0$, consider the group elements
\begin{align*}
	g_1 &\defeq \exp ( (\eps-1) X_1), &
	g_2 &\defeq \exp ( \eps ( X_2-X_1)), &
	g_3 &\defeq \exp ( (\tfrac{1}{2}-\eps ) X_2),\\
	g_4 &\defeq\exp( -\eps^2 X_1), & 
	g_5 &\defeq\exp( \tfrac{1}{2} X_2), &
	g_6 &\defeq\exp( \eps^2 X_1).
\end{align*}
Using the Baker-Campbell-Hausdorff Formula, which in step 2 is $\exp(X)\exp(Y) = \exp (X+Y +\tfrac{1}{2}[X,Y])$, one can verify that $\exp(X_2) = \exp(X_1) \, g_1\cdots g_6$.
We may assume that $|X_1|=| X_2|=1$.
Since $X_1$ and $X_2$ are linearly independent and the norm is strictly convex, the distance 
\[
D\defeq d( e, \exp (X_2-X_1))=\abs{X_2-X_1}
\]
is strictly smaller than 2.
By left-invariance of the distance and the triangle inequality, we get the upper bound
\[ 
d(\exp(X_1),\exp(X_2)) = d(e,g_1\cdots g_6) \leq \sum_{j=1}^6 d(e, g_j),
\]
which we can explicitly calculate as
\begin{align*}
\sum_{j=1}^6 d(e, g_j)
&= (1-\eps) + \eps D +(\tfrac{1}{2} - \eps) +\eps^2 +\tfrac{1}{2} +\eps^2\\
&= 2- (2-D) \eps + 2 \eps^2.
\end{align*}
Since $\lowestordercoeffcondition$, taking small enough $\eps>0$ we deduce $d( \exp (X_1), \exp (X_2))<2$.
Hence the corner from $ \exp (X_1)$ to $ \exp (X_2)$ is not length minimizing in $G$.
\end{proof}

The geometric interpretation of the next lemma is the following.
Curves from a quotient group can be isometrically lifted. 
Thus in our inductive argument we can use a geodesic from the previous step to get a curve that is shorter than the corner and has an error only in the last layer.

\begin{lemma}
\label{lem:lift:induction}
Let $G$ be a Carnot group of step $s$.
Assume that there are no minimizing corners in any Carnot group of step $s-1$ with first layer isometric to the first layer of $G$.
For all linearly independent $X_1, X_2\in V_1$ there exists $h\in \exp(V_s)$ such that 
\[
d( h \exp (X_1), \exp (X_2))<|X_1|+| X_2|.
\]
\end{lemma}
\begin{proof}
Consider the closed central subgroup $H\defeq\exp(V_s)$. The quotient $\lquot$ is a Carnot group of step $s-1$ with first layer $\pi_*(V_1)$.
Note that the norm on $\pi_*(V_1)$ is exactly the one that makes the projection $\pi_*:V_1\to \pi_*(V_1)$ an isometry. Therefore the first layer $\pi_*(V_1)$ of $\lquot$ is isometric to $V_1$, so by assumption there are no minimizing corners in $\lquot$.

If $X_1$ and $X_2$
are linearly independent, then so are $\pi_*(X_1)$ and $\pi_*(X_2)$.
Thus, by assumption, the corner in $\lquot$ from $\exp(\pi_*(X_1))$ to $\exp( \pi_*(X_2))$ is not length minimizing. Observe that since $\pi$ is a Lie group homomorphism, we have
$\exp(\pi_*(X))= \pi(\exp(X))$.
Hence, 
\[
d( \pi( \exp (X_1)), \pi( \exp (X_2)))<|X_1|+| X_2|.
\]
Using left-invariance of the distance on $G$ we see that
\begin{align*}
d( \pi( \exp (X_1)), \pi( \exp (X_2)))&= d( H\exp (X_1), H \exp (X_2))\\
&= \inf_{h\in H}d( h \exp (X_1), \exp (X_2)).
\end{align*}
Combining the above equality with the previous inequality, we conclude that there exists a point $h\in H$ for which the statement of the lemma holds.
\end{proof}

The next lemma is the technical core of our argument.
It shows that any error coming from Lemma~\ref{lem:lift:induction} can be corrected
using vectors in the layer $s-1$.
It also quantifies 
how the corrections change when scaling the error.
In what follows, we consider the conjugation map
$\conjg{p}(q) = pqp^{-1}$.

\begin{lemma}\label{lem:technic}
	Let $\cgp$ be a Carnot group of step $\step\geq 3$ and let $X_1$ and $X_2$ be vectors spanning $\calglayer{1}$. Then for any $\verterr\in \exp(\calglayer{s})$ there exist vectors $\perturbvec_1,\perturbvec_2,\perturbvec_3\in\calglayer{s-1}$ such that 
	\[ 
	\expconjg{X_1}{\epsilon^s\perturbvec_1}\conjgdot \expconjg{\frac{1}{2}X_2}{\epsilon^s\perturbvec_2}\conjgdot \expconjg{X_2}{\epsilon^s\perturbvec_3} = \dil{\epsilon}(\verterr),
	\]
	for all $\epsilon>0$.
\end{lemma}
\begin{proof}
Consider first for some $\errvec\in \calglayer{s}$ the equation
\begin{equation}\label{eq:triconjugate:lemma:triconjugate}
\expconjg{X_1}{\perturbvec_1}\conjgdot\expconjg{\frac{1}{2}X_2}{\perturbvec_2}\conjgdot\expconjg{X_2}{\perturbvec_3}=\exp(Z)
\end{equation}
in the variables $\perturbvec_1,\perturbvec_2,\perturbvec_3\in\calglayer{\step-1}$.
Since the step of the group $\cgp$ is $\step$, each conjugation can be expanded by the Baker-Campbell-Hausdorff 
Formula\footnote{Alternatively, one can use the formula \cite[page 114]{Warner}
\[ \expconjg{X}{Y} = \exp(\Ad_{\exp(X)}Y) = \exp(e^{\ad_X}Y). \]}
as 
\[ 
\expconjg{X}{Y} = \exp(X)\exp(Y)\exp(-X) = \exp(Y+\brkt{X}{Y}).
\]
We remark that the subgroup $\exp(\calglayer{s-1}\oplus\calglayer{s})$, containing the above conjugations, is commutative because of the assumption $\step\geq 3$. Hence $\exp$ is a homomorphism on $\calglayer{s-1}\oplus\calglayer{s}$. Consequently, since $\exp$ is also injective, we see that~\eqref{eq:triconjugate:lemma:triconjugate}
 is equivalent to the linear equation
\begin{equation}\label{eq:triconjugate lemma:linear}
\perturbvec_1+\perturbvec_2+\perturbvec_3+\brkt{X_1}{\perturbvec_1}+\brkt{X_2}{\tfrac{1}{2}\perturbvec_2+\perturbvec_3} = Z.
\end{equation}
Since the vectors $X_1$ and $X_2$ span the first layer $\calglayer{1}$, and $\calglayer{s}=\brkt{\calglayer{1}}{\calglayer{s-1}}$, for any $Z\in V_s$ there exist $\errveccomp_1,\errveccomp_2\in\calglayer{s-1}$ such that
\[ 
\errvec=\brkt{X_1}{\errveccomp_1}+\brkt{X_2}{\errveccomp_2}.
\]
Therefore, to solve the linear equation~\eqref{eq:triconjugate lemma:linear}, it is sufficient to solve the linear system
\begin{align*}
\perturbvec_1+\perturbvec_2+\perturbvec_3&=0\\
\perturbvec_1&=\errveccomp_1\\
\smash{\tfrac{1}{2}}\perturbvec_2+\perturbvec_3&=\errveccomp_2,
\end{align*}
which has the solution $
\perturbvec_1=\errveccomp_1,
\perturbvec_2=-2\errveccomp_1-2\errveccomp_2, \perturbvec_3=\errveccomp_1+2\errveccomp_2$.
Hence for any data $\errvec\in\calglayer{s}$, equation~\eqref{eq:triconjugate:lemma:triconjugate} has a solution $\perturbvec_1,\perturbvec_2,\perturbvec_3\in\calglayer{s-1}$.

Consider a fixed $\verterr\in\exp(\calglayer{s})$ and let $\errvec\in\calglayer{s}$ be such that $\exp(\errvec)=\verterr$. Note that then $\dil{\epsilon}(\verterr)=\exp(\epsilon^s\errvec)$ for any $\epsilon>0$. Recalling that the solution $\perturbvec_1,\perturbvec_2,\perturbvec_3$ for the data $\errvec$ is given by a linear equation, we have that for any $\epsilon>0$ the vectors $\epsilon^s\perturbvec_1,\epsilon^s\perturbvec_2,\epsilon^s\perturbvec_3$ give a solution for the data $\epsilon^s\errvec$, resulting in the statement of the lemma.
\end{proof}
\section{The main result}
\subsection{Reduction to Carnot groups}
The proof of Theorem~\ref{thm:mflds} can be reduced to the corresponding result for Carnot groups. Due to the possibility of the manifold not being equiregular (see \cite{jeancontrol} for the definition), we first consider a desingularization of the manifold near the corner-type singularity. Then we perform a blow-up, giving a corner in the metric tangent, which is a Carnot group by Mitchell's Theorem.

Let $M$ be a \subriemannian manifold with subbundle $\Delta$, and let $\gamma$ be a curve in $M$.
Fix a local orthonormal frame $X_1,\dots,X_r$ for $\Delta$ near $\gamma(0)$. By \cite[Lemma 2.5, page 49]{jeancontrol} there exists an equiregular \subriemannian manifold
$N$ with an orthonormal frame $\xi_1,\dots,\xi_r$ and a map $\pi: N \to M$ onto a neighborhood of $\gamma(0)$ such that $\pi_*\xi_i=X_i$.
We observe that $\pi $ is 1-Lipschitz with respect to the \subriemannian distances.

Assume that $\gamma$ is length minimizing, has a corner-type singularity at $0$, and is contained in $\pi(N)$. Let $u_j$ be integrable functions such that $\dot\gamma = \sum_j u_jX_j$ almost everywhere.

Let $\sigma$ be a curve in $N$ such that $\dot{\sigma} = \sum_j u_j\xi_j$ almost everywhere. Hence $\pi \circ \sigma ={\gamma} $ and the two curves $\sigma $ and ${\gamma} $ have the same length, see the proof of \cite[Lemma 2.5, page 49]{jeancontrol}. Since $\pi$ does not stretch distances, we conclude that $\sigma$ is length minimizing.
 
Since the vector fields $X_j$ form a frame, the coefficients $u_j$ are uniquely determined from $\dot{\gamma}$, and the existence of the left and right derivatives at 0 is equivalent to 0 being a left and right Lebesgue point for $u_j$. Therefore $\sigma$ also admits\footnote{We remark that for rank-varying distributions, desingularizations of curves with corner-type singularities need not have one-sided derivatives.} left and right derivatives at 0. Noting that $\pi_*\dot{\sigma}=\dot{\gamma}$ and that $\gamma$ has a corner-type singularity at $0$, we conclude that $\sigma$ also has a corner-type singularity at 0.

The curve $\sigma$ is now a length-minimizing curve with a corner-type singularity on an equiregular \subriemannian manifold $N$. The metric tangent of $N$ is a Carnot group $\cgp$, see a detailed proof in \cite[Proposition 2.4, page 39]{jeancontrol}. The blow-up of $\sigma$ on the Carnot group $\cgp$ is length minimizing and is given by the concatenation of two half-lines, see \cite[Proposition 2.4]{Leonardi-Monti}.

\subsection{The inductive non-minimality argument}

By the previous argument, to show that a length-minimizing curve in a \subriemannian manifold cannot have a corner-type singularity, it suffices to prove the corresponding result for Carnot groups. In fact, we prove the slightly stronger statement:
\begin{theorem}\label{thm:Carnot}
Corners are not length minimizing in any Carnot group equipped with a Carnot-Carath\'eodory distance coming from a strictly convex norm.
\end{theorem}

In the above, the distance is only coming from a strictly convex norm, as opposed to an inner product as in the \subriemannian case. The argument at the beginning of this section is however not dependent on the chosen distance. Thus it shows that Theorem~\ref{thm:mflds} also holds for C-C spaces with strictly convex norms.

\begin{proof}[Proof of Theorem~\ref{thm:Carnot}]
We remark that it  suffices to consider the case of rank-2 Carnot groups. Indeed, any corner is contained in some rank-2 subgroup, and if a curve is length minimizing, it must also be length minimizing in any subgroup containing it.
The theorem will then be proven for rank-2 Carnot groups by induction on the step $s$ of the group.
The base of induction is the case $s=2$, where the result is verified by Lemma~\ref{lem:starting:induction}.

Let $G$ be a rank-2 Carnot group of step $s$ with a Carnot-Carath\'eodory distance coming from a strictly convex norm. Consider the corner from $\exp(X_1)$ to $\exp(X_2)$, for some linearly independent $X_1,X_2\in\calglayer{1}$ with $\abs{X_1}=\abs{X_2}=1$.

Taking the quotient of $G$ by the central subgroup $\exp(\calglayer{s})$, we get a Carnot group of step $s-1$ whose first layer is isometric to the first layer of $G$. Note that the projection of a corner is still a corner in the quotient, where by induction we assume that corners are not length minimizing. Hence, by Lemma~\ref{lem:lift:induction}, there exists $h\in \exp(\calglayer{s})$ such that 
\begin{equation}\label{eq:thmcarnot:geodesiclift}
d( h \exp (X_1), \exp (X_2))<2.
\end{equation}
By Lemma~\ref{lem:technic}, for this fixed $h\in\exp(\calglayer{s})$, there exist vectors $Y_1, Y_2, Y_3\in\calglayer{s-1}$ satisfying the equation
\begin{equation}\label{eq:thmcarnot:triconjugate}
\dil{\epsilon}(h)^{-1}\expconjg{X_1}{\epsilon^sY_1}\conjgdot \expconjg{\frac{1}{2}X_2}{\epsilon^sY_2}\conjgdot \expconjg{X_2}{\epsilon^sY_3} = e.
\end{equation}
For a given $\eps>0$, consider the following points
\begin{align*}
g_1 &\defeq \exp(\epsilon^s\perturbvec_1) = \dil{\epsilon^{s/(s-1)}}(\exp(\perturbvec_1)),\\
g_2 &\defeq \exp(-(1-\epsilon)X_1)=\dil{1-\eps}(\exp(-X_1)),\\
g_3 &\defeq \exp(-\epsilon X_1) \dil{\epsilon}(\verterr)^{-1} \exp(\epsilon X_2)=\dil{\epsilon}\left(\exp(- X_1)\verterr^{-1}\exp(X_2) \right),\\
g_4 &\defeq \exp((\tfrac{1}{2}-\epsilon)X_2)=\dil{\tfrac{1}{2}-\epsilon}(\exp(X_2)),\\
g_5 &\defeq \exp(\epsilon^s\perturbvec_2) = \dil{\epsilon^{s/(s-1)}}(\exp(\perturbvec_2)),\\
g_6 &\defeq \exp(\tfrac{1}{2}X_2)=\dil{\tfrac{1}{2}}(\exp(X_2)),\quad\text{and}\\
g_7 &\defeq\exp(\epsilon^s\perturbvec_3) = \dil{\epsilon^{s/(s-1)}}(\exp(\perturbvec_3)).
\end{align*}
We claim that
\begin{equation}\label{eq:thmcarnot:endpt}
\exp(X_2) = \exp(X_1) \, g_1\cdots g_7,
\end{equation}
and that for small enough $\epsilon>0$
\begin{equation}\label{eq:thmcarnot:pointdistance}
\sum_{j=1}^7 d(e, g_j) <2,
\end{equation}
from which the result of the theorem will follow.
Regarding \eqref{eq:thmcarnot:endpt}, writing explicitly the definitions of the points $g_j$, we have
\begin{align*}
\exp(X_1) \, g_1\cdots g_7
= \exp(X_1)
&\exp(\epsilon^s\perturbvec_1) 
\exp(-(1-\eps)X_1)
\exp(-\epsilon X_1) \dil{\epsilon}(\verterr)^{-1}\\
\cdot&\exp(\epsilon X_2)
\exp((\tfrac{1}{2}-\epsilon)X_2)
\exp(\epsilon^s\perturbvec_2) 
\exp(\tfrac{1}{2}X_2)
\exp(\epsilon^s\perturbvec_3).
\end{align*}
Then, using the fact that $h$ is in $Z(G)$, we rewrite the right-hand side in terms of conjugations as
\[
\delta_\eps (h)^{-1} 
\expconjg{X_1}{\epsilon^sY_1}\conjgdot \expconjg{\frac{1}{2}X_2}{\epsilon^sY_2}\conjgdot \expconjg{X_2}{\epsilon^sY_3}\exp(X_2).
\]
Since 
$Y_1, Y_2, Y_3$ were chosen to satisfy~\eqref{eq:thmcarnot:triconjugate}, the above term reduces to $\exp(X_2)$, thus showing~\eqref{eq:thmcarnot:endpt}.
To show~\eqref{eq:thmcarnot:pointdistance}, we note that as the points $g_j$ are all dilations of some fixed points, the individual distances are given by
\begin{align*}
d(e,g_1) &= \epsilon^{s/(s-1)} d(e,\exp(\perturbvec_1)),\\
d(e,g_2) &= 1-\epsilon\\
d(e,g_3) &= \epsilon d(e, \exp(- X_1)h^{-1}\exp(X_2)) = \epsilon d(h\exp(X_1), \exp(X_2)),\\
d(e,g_4) &= \frac{1}{2}-\epsilon,\\
d(e,g_5) &= \epsilon^{s/(s-1)} d(e,\exp(\perturbvec_2)),\\
d(e,g_6) &= \frac{1}{2}\quad\text{and}\\
d(e,g_7) &= \epsilon^{s/(s-1)} d(e,\exp(\perturbvec_3)).
\end{align*}
Summing all the above distances, we get
\[
\sum_{j=1}^7 d(e, g_j) 
= 2 - (2-D) \epsilon +o(\epsilon),\quad\text{as }\epsilon\to 0,
\]
where 
\[D\defeq d( h \exp (X_1), \exp (X_2)).\]
By the choice of $h$ from~\eqref{eq:thmcarnot:geodesiclift},
we have $\lowestordercoeffcondition$.
Therefore, for small enough $\eps>0$, we deduce~\eqref{eq:thmcarnot:pointdistance}.

We finally estimate using left-invariance, equations~\eqref{eq:thmcarnot:endpt} and~\eqref{eq:thmcarnot:pointdistance}, and the triangle inequality, that
\[
d(\exp(X_1), \exp(X_2))
=d( e, g_1\cdots g_7)
\leq \sum_{i=1}^7 d(e, g_i)<2,
\]
for small enough $\epsilon>0$.
Since the considered corner from $\exp(X_1)$ to $\exp(X_2)$ has length equal to 2, where $X_1$ and $X_2$ were arbitrary linearly independent unit-norm vectors of the first layer $\calglayer{1}$, we conclude that corners in the group $G$ of step $s$ are not length minimizing.
\end{proof}

\subsubsection{Acknowledgement}
The authors thank A.~Ottazzi, D.~Vittone, and the anonymous referees for their helpful remarks.
E.L.D. acknowledges the support of the Academy of Finland project no. 288501. 

\medskip 

A S\'eminaire Bourbaki presentation including the content of this paper and related work has been subsequently given by L.~Rifford  \cite{Rifford:Bourbaki}.

\bibliography{general_bibliography-corners}
\bibliographystyle{amsalpha}
\end{document}